\documentclass[11pt]{amsart}

\DeclareSymbolFont{symbolsC}{U}{pxsyc}{m}{n}
\DeclareMathSymbol{\coloneqq}{\mathrel}{symbolsC}{"42}

\def\R{{\mathbb {R}}}
\def\N{{\mathbb {N}}}

\def\A{{\mathcal {A}}}
\def\B{{\mathcal {B}}}
\def\B{{\mathcal {B}}}
\def\X{{\mathcal {X}}}
\def\M{{\mathcal {M}}}
\def\n{{\mathbf {n}}}

\def\d{\operatorname{\text{dist}}}

\newcommand{\p}{\partial}
\def\eps{{\varepsilon}}

\newtheorem{teo}{Theorem}[section]
\newtheorem{lema}[teo]{Lemma}
\newtheorem{prop}[teo]{Proposition}
\newtheorem{corol}[teo]{Corollary}

\theoremstyle{remark}
\newtheorem{remark}[teo]{Remark}

\theoremstyle{definition}

\numberwithin{equation}{section}

\begin{document}

\title[Gamma convergence]{A Gamma convergence approach to the critical Sobolev embedding in variable exponent spaces}

\author[J. Fern\'andez Bonder, N. Saintier and A. Silva]{Juli\'an Fern\'andez Bonder, Nicolas Saintier and Analia Silva}

\address[J. Fern\'andez Bonder and A. Silva]{IMAS - CONICET and Departamento de Matem\'atica, FCEyN - Universidad de Buenos Aires, Ciudad Universitaria, Pabell\'on I  (1428) Buenos Aires, Argentina.}

\address[N. Saintier]{Instituto de Ciencias - Univ. Nac. Gral Sarmiento, J. M. Gutierrez 1150, C.P. 1613 Los Polvorines - Pcia de Bs. As. - Argentina and Dpto Matem\'atica, FCEyN - Univ. de Buenos Aires, Ciudad Universitaria, Pabell\'on I  (1428) Buenos Aires, Argentina.}

\email[J. Fernandez Bonder]{jfbonder@dm.uba.ar}

\urladdr[J. Fernandez Bonder]{http://mate.dm.uba.ar/~jfbonder}

\email[A. Silva]{asilva@dm.uba.ar}

\email{nsaintie@dm.uba.ar, nsaintie@ungs.edu.ar}

\urladdr[N. Saintier]{http://mate.dm.uba.ar/~nsaintie}



\begin{abstract}
In this paper, we study the critical Sobolev embeddings $W^{1,p(x)}(\Omega)\subset L^{p^*(x)}(\Omega)$ for variable exponent Sobolev spaces from the point of view of the $\Gamma$-convergence. More precisely we determine the $\Gamma$-limit of subcritical approximation of the best constant associated with this embedding. As an application we provide a sufficient condition for the existence of extremals for the best constant. 
\end{abstract}

\maketitle

\section{Introduction}

The purpose of this paper is to analyze the Sobolev immersion theorem for variable exponent spaces in the critical range from the point of view of the $\Gamma$-convergence. Our motivation comes from the existence problem for extremals of these immersions. By extremals we mean functions $u\in W^{1,p(x)}_0(\Omega)$ where the  infimum 
\begin{equation}\label{S}
S = S(p(\cdot), q(\cdot),\Omega) \coloneqq  \inf_{v\in W^{1,p(x)}_0(\Omega)} \frac{\| \nabla v \|_{p(x)}}{\|v\|_{q(x)}}
\end{equation}
is attained. We refer to the next section for the definition of the variable exponent Sobolev spaces and the norms appearing in (\ref{S}). We shall assume  the set $\A\coloneqq \{ x\in\Omega\colon q(x) = p^*(x)\}$  non-empty (here $p^*(x)$ is the Sobolev conjugate of $p(x)$, see next section) so that the problem of existence of an extremal for $S$ is {\it critical} from the Sobolev embedding point of view. 

This problem was recently treated in \cite{FBSS1} where the authors provide sufficient conditions to ensure the existence of such extremals. The approach in \cite{FBSS1} was the so-called {\em direct method of the calculus of variations}. That is, they considered a minimizing sequence for $S$ and find a sufficient condition that ensured the compactness of such sequence.

In this paper, we follow a different approach. Instead of looking for minimizing sequences for $S$, we approximate the critical problems by subcritical ones, where the existence of extremals is easily obtained, and then pass to the limit. In fact, following G. Palatucci in \cite{Palatucci} where  the constant exponent case is studied, we want  to determine the asymptotic behaviour in the sense of the $\Gamma$-convergence of the subcritical approximations 
$$ 
S_\eps \coloneqq  \inf_{v\in W^{1,p(x)}_0(\Omega)} \frac{\| \nabla v \|_{p(x)}}{\|v\|_{q(x)-\eps}},\quad \eps>0,
$$
and then deduce the behavior of their associated extremals $u_\eps$.

\subsection{Preliminary notations} 

Let $\Omega$ be  smooth open bounded subset of $\R^n$. Given a measurable function $p\colon \Omega\to [1,+\infty)$, the Lebesgue variable exponent space $L^{p(x)}(\Omega)$ is defined as
$$ 
L^{p(x)}(\Omega) \coloneqq  \left\{ u\in L^1_{\text{loc}}(\Omega)\colon \int_\Omega |u|^{p(x)}\, dx < +\infty\right\}. 
$$
This space is endowed with the norm
$$ 
\|u\|_{p(x)} \coloneqq  \inf\left\{\lambda>0\colon \int_\Omega \Big|\frac{u}{\lambda}\Big|^{p(x)}\, dx \le 1 \right\}. 
$$
which turns $L^{p(x)}(\Omega)$  into a Banach space. Assuming moreover that 
\begin{equation}\label{p+-}
1< p_- \coloneqq  \inf_\Omega p \le p_+ \coloneqq \sup_\Omega p <+\infty, 
\end{equation}
it can be proved that $L^{p(x)}(\Omega)$ separable and reflexive. 

\medskip

These spaces where first considered in the seminal W. Orlicz' paper \cite{Orlicz} in 1931 but then where left behind as the author pursued the study of the spaces that now bear his name. The first systematic study of these spaces appeared in H. Nakano's  works at the beginning of the 1950s \cite{Nakano1, Nakano2} where he developed a general theory in which the spaces $L^{p(x)}(\Omega)$ were a particular example of the more general spaces he was considering. Even though some progress was made after Makano's work (see in particular the works of the Polish school H. Hudzik, A. Kami\'nska and J. Musielak in e.g. \cite{Hudzik, Kaminska, Musielak}), it was only in the last 20 years that major progress has been accomplished mainly due to the following facts:
\begin{itemize}
\item The discovery of a very weak condition ensuring the boundedness of the Hardy-Littlewood maximal operator in these spaces, i.e. the log-H\"older condition that implies, to begin with, that test functions are dense in $L^{p(x)}(\Omega)$.

\item The discovery of the connection of these spaces with the modeling of the so-called {\em electrorheological fluids} \cite{Rusika}

\item The application that variable exponents have shown in image processing \cite{CLR}
\end{itemize}

\medskip

Of central importance in the above mentioned applications are the variable exponent Sobolev spaces $W^{1,p(x)}(\Omega) $   defined as
$$ 
W^{1,p(x)}(\Omega) \coloneqq  \left\{ u\in W^{1,1}_{\text{loc}}(\Omega)\colon u, \partial_i u\in L^{p(x)}(\Omega)\ i=1,\dots,n\right\}, 
$$
and the subspace of functions with zero boundary values
$$ 
W^{1,p(x)}_0(\Omega) \coloneqq \overline{\{u\in W^{1,p(x)}(\Omega)\colon u \text{ has compact support}\}}, 
$$
where the closure is taken in the $W^{1,p(x)}(\Omega)-$norm $\|\cdot\|_{1,p(x)}$  that is defined as
$$ 
\|u\|_{1,p(x)} \coloneqq \|u\|_{p(x)} + \| \nabla u \|_{p(x)}. 
$$

We assume from now on that $p$ is log-H\"older in the sense that 
\begin{equation}\label{log}
\sup_{x,y\in \Omega} |(p(x)-p(y)) \log(|x-y|)| <+\infty. 
\end{equation} 
Under this assumption it can be proved  that the space $C^\infty_c(\Omega)$ is dense in $L^{p(x)}(\Omega)$  and in $W^{1,p(x)}_0(\Omega)$, and also that the Poincar\'e inequality holds i.e. there exists a constant $C=C(\Omega,p)>0$ such that
$$
\|u\|_{p(x)} \le C \| \nabla u \|_{p(x)}
$$
for any $u\in W^{1,p(x)}_0(\Omega)$. It follows in particular that $\|\nabla u\|_{p(x)}$ is an equivalent norm in $W^{1,p(x)}_0(\Omega)$.

\subsection{Critical Sobolev embedding}

A major tool in order to study existence and regularity properties of solutions to partial differential equatios is the Sobolev embedding theorem. For variable exponents spaces this theorem has been established in \cite{Ko-Ra} (see also \cite{Fan}). Given a measurable function $q\colon \Omega\to [1,+\infty)$, it basically says that we have a continuous embedding 
$$ 
W^{1,p(x)}_0(\Omega)\subset L^{q(x)}(\Omega) 
$$
if and only if $q(x)\le p^*(x) \coloneqq  np(x)/(n-p(x))$. Moreover, when the exponent $q$ is {\em strictly subcritical} in the sense that 
$$ 
\inf_{x\in \Omega} (p^*(x)-q(x))>0, 
$$
then this embedding is compact (see e.g. \cite{Diening}). On the other hand when the {\it critical set} 
\begin{equation}\label{CriticalSet}
\A \coloneqq  \{x\in \bar\Omega\colon q(x)=p(x)^*\}
\end{equation} 
is  not empty, the immersion is no longer compact in general (see \cite{MOSS} for some very restricted cases where $\A\neq\emptyset$ but the immersion still remains compact). The existence of extremal for the best constant $S$ defined in (\ref{S}) is then not granted. Indeed the well-known Pohozaev identity implies that  when $p$ is constant and $\Omega $ is star-shaped then $S$ is not attained. 

The problem of existence of extremals for $S$ in the variable exponent setting was recently considered in \cite{FBSS1} where the authors provided sufficient existence conditions. A fundamental tool used in their proof, as well as in almost every problem dealing with critical expoenent in general, is the so called {\em Concentration Compactness Principle} (CCP) that was introduced by P. L. Lions in the 80's (see \cite{Lions}) and was recently extended to the variable exponent setting in \cite{FBS} (see also the refinement in \cite{FBSS1}). This version of the CCP relies on a notion of localized Sobolev constant defined as follows. 
For $x\in\A$ we define the localized best Sobolev constant $\bar S_x$ as
\begin{equation}\label{defSobLoc}
\bar S_x \coloneqq \lim_{\eps\to 0} S(p(\cdot),q(\cdot),B_\eps(x)\cap\Omega). 
\end{equation}
Notice that
\begin{equation}\label{Sp<S}
S(p(\cdot), q(\cdot), \Omega) \le \inf_{x\in \A} \bar S_x . 
\end{equation}

The CCP proved in \cite{FBS} and refined in \cite{FBSS1} states that given a weakly convergent sequence $\{u_k\}_{k\in\N}\subset W^{1,p(x)}_0(\Omega)$ with weak limit $u$, there exists a countable set of indices $I$, positive real numbers $\{\mu_i\}_{i\in I}, \{\nu_i\}_{i\in I}\subset \R_+$, points $\{x_i\}_{i\in I}\in \A$ and nonnegative measures $\mu$, $\nu$ such that
\begin{align} 
|u_k|^{q(x)}\, dx &\stackrel{*}{\rightharpoonup} d\nu = |u|^{q(x)}\, dx + \sum_{i\in I} \nu_i\,  \delta_{x_i}, \label{CCP}\\
|\nabla u_k|^{p(x)}\, dx &\stackrel{*}{\rightharpoonup} d\mu \ge |\nabla u|^{p(x)}\, dx + \sum_{i\in I} \mu_i\,  \delta_{x_i},\label{CCP2}\\
 \bar S_{x_i} \nu_i^{\frac{1}{p(x_i)^*}} &\le \nu_i^{\frac{1}{p(x_i)}} \quad \text{ for any } i\in I. \label{CCP3}
\end{align}
It is also easily checked that the nonnegative measure
$$ 
\tilde\mu\coloneqq  \mu - \left(|\nabla u|^{p(x)}\, dx + \sum_{i\in I} \mu_i\,  \delta_{x_i}\right) 
$$
has no atoms.

By analyzing the behavior of minimizing sequences using the CCP, it is proved in \cite{FBSS1} that if the inequality in \eqref{Sp<S} is strict, $q_-<p_+$ and $p,q$ are slightly more regular than merely Log-H\"older continuous,  more precisely if
\begin{equation}\label{logholder+}
\lim_{\Omega\ni y\to x} (p(y)-p(x)) \log(|x-y|) = \lim_{\Omega\ni y\to x} (q(y)-q(x)) \log(|x-y|) = 0, \text{ uniformly in } x\in \Omega.
\end{equation}
is satisfied, then there exists an extremal for $S(p(\cdot), q(\cdot), \Omega)$, i.e. a function $u\in W^{1,p(x)}_0(\Omega)$ where the infimum in \eqref{S} is attained.

\subsection{Statements of the results}

The main purpose of this paper is to study the limit as $\eps\to 0$ of the subcritical approximation 
$$ 
S_\eps \coloneqq  S(p(\cdot), q(\cdot)-\eps, \Omega) = \inf_{v\in W^{1,p(x)}_0(\Omega)} \frac{\| \nabla v \|_{p(x)}}{\|v\|_{q(x)-\eps}},\quad \eps>0,
$$
of  $S(p(\cdot), q(\cdot), \Omega)$ following the work of  \cite{Amar-Garroni} and \cite {Palatucci}.  This amounts to study the asymptotic behaviour as $\eps\to 0$ of the functional $F_\eps: \B(\Omega)\to \R$ defined by
$$ 
F_\varepsilon(u) \coloneqq  \int_\Omega|u|^{q(x)-\varepsilon}\, dx, 
$$
where 
\begin{equation}\label{DefB}  
\B(\Omega) \coloneqq  \Big\{u\in W^{1,p(x)}_0(\Omega),\, \|\nabla u\|_{p(x),\Omega}\le 1\Big\}. 
\end{equation} 
This  will be done in the framework of the $\Gamma$-convergence. 

In view of the CCP (\ref{CCP})-(\ref{CCP3}), it turns out  to be convenient to extend $F_\eps$ to the space 
\begin{align*}
\X=\X(\Omega)=\Big\{(u,\mu)\in W_0^{1,p(x)}(\Omega)\times \M(\overline{\Omega})\colon  \mu(\overline{\Omega})\le 1,\ \mu = |\nabla u|^{p(x)}\, dx + \tilde\mu + \sum_{i\in I} \mu_i \delta_{x_i}\Big\}, 
\end{align*}
where $\M(\overline{\Omega})$ is the space of bounded measures over $\overline{\Omega}$ and, in the decomposition of $\mu$, $\tilde\mu$ is a nonnegative measure without atoms, the set $I$ is at most countable, the scalars $\mu_i$ are positive, and the atoms $x_i$ belongs to the critical set $\A$ defined in \eqref{CriticalSet}.
 
We say that a sequence $\{(u_\varepsilon,\mu_\varepsilon)\}_{\eps>0}\subset \X$ converges in $\X$ to $(u,\mu)$, which is denoted by $(u_\eps, \mu_\eps)\stackrel{\tau}{\to} (u,\mu)$, if  $u_\varepsilon\rightharpoonup u$ weakly in $L^{q(x)}(\Omega)$ and $\mu_\varepsilon\stackrel{*}{\rightharpoonup} \mu$ in $\mathcal{M}(\overline{\Omega})$. We recall that $\mu_\varepsilon\stackrel{*}{\rightharpoonup} \mu$ means that $\int\phi\,d\mu_\eps\to\int\phi\,d\mu$ for any $\phi\in C_b(\overline{\Omega})$. 

We then extend $F_\varepsilon$ to the whole space $\X$ by
$$ 
F_\varepsilon(u,\mu) = \begin{cases} 
\int_\Omega|u|^{q(x)-\varepsilon}\,dx  \quad \text{ if } d\mu=|\nabla u|^{p(x)}\,dx \\
0 \hspace{2.5cm}\mbox{ otherwhise in } \X
\end{cases}
$$

We also consider the limit functional $F^*\colon \X\to \R$ defined by
$$
F^*(u,\mu) \coloneqq  \int_\Omega|u|^{q(x)}\,dx+\sum_{i\in I}\mu_i^\frac{p^*(x_i)}{p(x_i)}\bar S^{-p^*(x_i)}_{x_i}.
$$
where $\bar S_{x_i}$, $i\in I$, is defined in (\ref{defSobLoc}). 

\medskip

Our main result is the following

\begin{teo}\label{TEO}
The functionals $\{F_\eps\}_{\eps>0}$ $\Gamma-$converge to $F^*$  in the sense that for any $(u,\mu)\in \X$ there holds that 
\begin{itemize}
\item for every  sequence $\{(u_\eps,\mu_\eps)\}_{\eps>0}\subset \X$ converging to $(u,\mu)$ in $\X$, we have
\begin{equation}\label{Limsup}
\limsup_{\eps\to 0} F_\eps(u_\eps,\mu_\eps) \le F^*(u,\mu)
\end{equation}
\item there exists a sequence $\{(u_\eps,\mu_\eps)\}_{\eps>0}\subset \X$ converging to $(u,\mu)$ in $\X$ such that
\begin{equation}\label{Liminf}
\liminf_{\eps\to 0} F_\eps(u_\eps,\mu_\eps) \ge F^*(u,\mu)
\end{equation}
\end{itemize}
\end{teo}

\begin{remark} What we called $\Gamma-$convergence in Theorem \ref{TEO} is what other authors called $\Gamma^+-$convergence. See \cite{Palatucci}.
\end{remark}

\begin{remark}
Consider the functionals $F^+,F^-\colon \X\to [0,+\infty)$ defined by
$$
F^+(u,\mu) \coloneqq \sup\,\{\limsup F_\varepsilon(u_\varepsilon,\mu_\varepsilon)\colon (u_\varepsilon,\mu_\varepsilon)\stackrel{\tau}{\to} (u,\mu)\},
$$
and
\begin{equation}\label{defF-}
F^-(u,\mu) \coloneqq \sup\,\{\liminf F_\varepsilon(u_\varepsilon,\mu_\varepsilon)\colon (u_\varepsilon,\mu_\varepsilon)\stackrel{\tau}{\to} (u,\mu)\}.
\end{equation} 
Obviously, $F^-\le F^+$. Moreover the $\Gamma$-limit $F^*$ of the sequence $\{F_\eps\}_{\eps>0}$ exists if and only if $F^-=F^+$ and in this case, $F^*=F^-=F^+$. Therefore, we can rewrite the previous theorem as
$$ 
F^+\le F^* \quad \text{ and } \quad  F^-\ge F^*. 
$$
\end{remark}

Define 
\begin{equation}\label{defTildeS}
 \tilde S^{-1} \coloneqq  \tilde S(p(\cdot), q(\cdot),\Omega)= \sup_{u\in\B(\Omega)}\int_\Omega |u|^{q(x)}\,dx,
\end{equation} 
where $\B(\Omega)$ is defined in (\ref{DefB}). We also define a local best constant $\tilde S^{-1}_{x_0}$, $x_0\in\A$, in a smiliar way as in (\ref{defSobLoc}) by 
\begin{equation}\label{LocSob}
\widetilde S_{x_0}^{-1} \coloneqq \lim_{\varepsilon\to0}\left(\sup_{u\in \B(B_\eps(x_0))}\int_{B_{\varepsilon}(x_0)}|u|^{q(x)}\,dx\right), \qquad x_0\in\A.
\end{equation} 
Noticing that $\B(B_{x_0}(\eps))\subset \B(\Omega)$, we have  that 
\begin{equation}\label{Cont}
 \sup_{x_0\in\A} \widetilde S_{x_0}^{-1} \le \widetilde S^{-1}. 
\end{equation} 
We also prove in lemma \ref{SobLoc2} below that $\widetilde S_{x_0}^{-1}=\bar S_{x_0}^{-q(x_0)}$. 

We now consider the subcritical approximations $\tilde S_\eps^{-1}$ of $\tilde S^{-1}$ defined by 
$$ 
\tilde S_\eps^{-1}\coloneqq \tilde S(p(\cdot), q(\cdot)-\eps, \Omega)^{-1} = \sup_{u\in \B(\Omega)}\int_\Omega |u|^{q(x)-\eps}\,dx. 
$$ 
We first prove that 

\begin{prop} \label{SeS} There holds that 
$$ 
\lim_{\eps\to 0} \tilde S_\eps^{-1} = \tilde S^{-1}. 
$$
\end{prop}
In the same spirit as in \cite[Theorem 4.2]{FBSS1}, we can deduce from the $\Gamma$-convergence of $F_\eps$ to $F$ the asymptotic  behaviour of extremals for $\tilde S_\eps$: 

\begin{teo}\label{u_eps}
Let $u_\eps\in \B(\Omega)$ be an extremal for $\tilde S_\eps^{-1}$, i.e.
$$ 
\int_\Omega |u_\eps|^{q(x)-\eps}\,dx = \tilde S_\eps^{-1}. 
$$
Then the following alternative hols:
\begin{enumerate}
\item either the sequence $\{u_\eps\}_{\eps>0}$ has a strongly convergent subsequence in $L^{q(x)}(\Omega)$ and the strong limit is an extremal for $\tilde S^{-1}$, 

\item or the sequence $\{u_\eps\}_{\eps>0}$ concentrates around a single point $x_0\in\mathcal{A}$ in the sense that 
$$
|u_\eps|^{q(x)}\, dx \rightharpoonup  \tilde S^{-1}\, \delta_{x_0} \quad \text{ and }\quad 
|\nabla u_\eps|^{p(x)}\, dx \rightharpoonup \delta_{x_0}.
$$
Moreover
\begin{equation}\label{Conc} 
\tilde S^{-1}_{x_0}=\tilde S^{-1}. 
\end{equation} 
\end{enumerate}
\end{teo}

As an immediate consequence of (\ref{Cont}) and (\ref{Conc}), we obtain the following sufficient condition for the existence of an extremal for $\tilde S^{-1}$: 

\begin{corol}\label{existencia}
If $\sup_{x\in \A} \tilde S^{-1}_x<\tilde S^{-1}$, then any sequence of extremals for $\tilde S^{-1}_\eps$ converges, up to a subsequence, to some $u\in \B(\Omega)$ which is an extremal for $\tilde S^{-1}$. In particular, there exists an extremal for $\tilde S^{-1}$.
\end{corol}

\noindent This kind of sufficient condition of existence is common in the study of problems with critical exponent. In the constant exponent case, it goes back to \cite{Aubin, BN}, and \cite{Lions}. In the  variable exponent case, it was recently established and used by the authors in \cite{FBSS1, FBSS2, FBSS3, FBSS4} where precise condition on the exponents $p$ and $q$ were provided for this condition to hold.

\section{Proof of Theorem \ref{TEO}}

We divided the proof into two subsections: one for the $\limsup$ inequality (\ref{Limsup}) and other for the $\liminf$ inequality (\ref{Liminf}). The strategy of the proof is completely analogous to that of \cite{Palatucci} where the constant exponent case is treated with difficulties specific to the variable exponent setting. 

\subsection{Proof of the $\limsup$ property \eqref{Limsup}}

Consider a sequence $\{(u_\eps,\mu_\eps)\}_{\eps>0} \subset \X$ converging as $\eps\to 0$ to some $(u_0,\mu_0)\in \X$.  We can assume without loss of generality that $d\mu_\eps = |\nabla u_\eps|^{p(x)}\,dx$ for all $\eps>0$. Then by H\"older inequality (see \cite[lemma 3.2.20]{Diening}):  
\begin{align*}
F_\varepsilon(u_\varepsilon,\mu_\varepsilon)
&=\int_\Omega|u_\varepsilon|^{q(\cdot)-\varepsilon}\,dx\\
&\leq \Big(\frac{1}{(\frac{q(x)}{q(x)-\varepsilon})^-} + \frac{1}{(\frac{q(x)}{\varepsilon})^-}\Big)
\||u_\varepsilon|^{q(x)-\varepsilon}\|_{\frac{q(x)}{q(x)-\varepsilon}}\|1\|_\frac{q(x)}{\varepsilon}. 
\end{align*}			
Since $\Big(\frac{q(x)}{q(x)-\varepsilon}\Big)^- \to 1$, $\Big(\frac{\eps}{q(x)}\Big)^- \to 0$, and $\|1\|_{\frac{q(x)}{\varepsilon}}\to 1$ as $\varepsilon\to 0$, we obtain 
$$ 
\limsup_{\eps\to 0}F_\varepsilon(u_\varepsilon,\mu_\varepsilon) \le \limsup_{\eps\to 0}     \||u_\varepsilon|^{q(x)-\varepsilon}\|_{\frac{q(x)}{q(x)-\varepsilon}}. 
$$  

Up to some subsequence, by the CCP, there exists $u\in W^{1,p(x)}_0(\Omega)$ and measures $\nu, \mu\in \M(\overline{\Omega})$ of the form $\nu = |u|^{q(x)} dx + \sum_{i\in I} \nu_i \delta_{x_i}$, $\mu = |\nabla u|^{p(x)} dx + \tilde{\mu} +\sum_{i\in I}\mu_i \delta_{x_i}$ such that
\begin{align*}
&u_\eps \rightharpoonup u \qquad \text{weakly in } W^{1,p(x)}_0(\Omega) \text{ and in } L^{q(x)}(\Omega),\\
&|u_\eps|^{q(x)}\, dx \stackrel{*}{\rightharpoonup} \nu \quad \text{and} 
\quad |\nabla u_\eps|^{p(x)}\, dx \stackrel{*}{\rightharpoonup} \mu.
\end{align*}
Observe that $u=u_0$ and $\mu_0 = \mu$ since $(u_\eps,\mu_\eps = |\nabla u_\eps|^{p(x)})\to (u_0,\mu_0)$ in $\X$. It follows that 
$$ 
\rho_{q(x)}(u_\varepsilon)\coloneqq \int_\Omega |u_\eps|^{q(x)}\,dx \to \int_\Omega|u_0|^{q(x)}\,dx+\sum_{i\in I} \nu_i \qquad 
\text{ as } \eps\to 0. 
$$  

Since  
$$ 
\||u_\varepsilon|^{q(x)-\varepsilon}\|_{\frac{q(x)}{q(x)-\varepsilon}}\leq \max\Big\{\rho_{q(x)}(u_\varepsilon)^{(\frac{q(x)}{q(x)-\varepsilon})^+},\rho_{q(x)}(u_\varepsilon)^{(\frac{q(x)}{q(x)-\varepsilon})^-}\Big\},
$$
we obtain, in view of \eqref{CCP3}, that 
\begin{align*}
\limsup_{\varepsilon\to 0^+}F_\varepsilon(u_\varepsilon,\mu_\varepsilon)&\leq  \int_\Omega |u_0|^{q(x)}\,dx+ \sum_{i\in I} S_{x_i}^{-p(x_i)} \mu_i^{\frac{p^*(x_i)}{p(x_i)}} = F^*(u_0,\mu_0).
\end{align*}
This is the limsup inequality.
\qed

\subsection{Proof of the $\liminf$ property \eqref{Liminf}}

We begin with an elementary inequality that will be most useful in the sequel. Though we believe that this inequality is well known, we were unable to find it in the literature.

\begin{prop} \label{desigualdad.nico}
Given $p>1$ and $\theta\in (0,1]$ there exists a positive constant $C>0$ such that
$$ 
\Big||a+b|^p-|a|^p-|b|^p\Big|\le C\Big( |a|^{p-\theta}|b|^\theta + |a|^\theta |b|^{p-\theta} \Big) 
$$
for any $a,b\in\R^n$.
\end{prop}

\begin{proof}
Since this inequality is invariant under rotation and dilatation of $a$ and $b$ we can assume that $b=e_1=(1,0,\dots,0)$. The function $f\colon\R^n\backslash\{0\}\to [0,+\infty)$ defined by
$$ 
f(a)=\frac{\displaystyle \Big||a+e_1|^p-|a|^p-1\Big|}{\displaystyle |a|^{p-\theta}+|a|^\theta} 
$$
being continuous, it suffices to prove that it remains bounded near $0$ and $\infty$ to obtain that it is bounded in all  $ \R^n\backslash\{0\}$. First
\begin{equation*}
\begin{split}
 |a+e_1|^p & = \Big( |a|^2+2ae_1+1\Big)^{p/2} \\
& = \begin{cases}
1+O(|a|)  \hspace{5.26cm} \text{ for } |a|\ll 1, \\
|a|^p (1+O(1/|a|))^{p/2} = |a|^p + O(|a|^{p-1})   \quad \text{ for } |a|\gg 1.
\end{cases}
\end{split}
\end{equation*}
It follows that for $|a|\ll 1$ we have
\begin{equation*}
\begin{split}
 f(a) & =\frac{O(|a|)}{\displaystyle |a|^{p-\theta}+|a|^\theta}  \\
        &   = \begin{cases}
                \frac{\displaystyle  O(|a|)}{\displaystyle  |a|^\theta (1+o(1))}
                      = O(|a|^{1-\theta})  \le C \quad \text{ if } \theta\le \frac{p}{2} \\
                \frac{\displaystyle  O(|a|)}{\displaystyle  |a|^{p-\theta}(1+o(1))}
                     = O(|a|^{1-p+\theta})  \quad \text{ if } \theta\ge \frac{p}{2}
              \end{cases}
\end{split}
\end{equation*}
Hence if $p\ge 2$ so that $\theta\le 1\le p/2$, or if $p<2$ and $\theta\le p/2$, then $f$ is bounded near $0$. Now if $p<2$ and $\theta\ge p/2$ then $\theta\ge p/2\ge p-1$ so that $f$ is also bounded near $0$ in that case.

For $|a|\gg 1$ we obtain
\begin{equation*}
\begin{split}
 f(a) & = \frac{\displaystyle O(|a|^{p-1})}{\displaystyle |a|^{p-\theta}+|a|^\theta}   \\
& =\begin{cases}
  O(|a|^{\theta-1}) \le C \quad \text{ if } \theta\le \frac{p}{2} \\
  O(|a|^{p-1-\theta})  \hspace{0.84cm} \text{ if } \theta\ge \frac{p}{2}
              \end{cases}
\end{split}
\end{equation*}
Hence if $p\ge 2$ so that $\theta\le 1\le p/2$, or if $p<2$ and $\theta\le p/2$, then $f$ is bounded near infinity. If $p<2$ and  $\theta\ge p/2$ then $\theta\ge p/2\ge p-1$ and the same conclusion holds.
\end{proof}

\begin{remark}\label{remark}
It can be easily checked, from the above proof, that for any $p\in [p_-, p_+]$, the constant $C$ can be taken depending only on $p_-, p_+$ and $\theta$.
\end{remark}

\begin{prop}\label{Exist}
For any $(u,\mu)\in \X$ there exists a sequence $(u_\eps,|\nabla u_\eps|^{p(x)}\,dx)\in \X$ converging in $\X$ to $(u,\mu)$ as $\eps\to 0$.
\end{prop}

\begin{proof} We adapt the proof of M. Amar and A. Garroni \cite{Amar-Garroni}. We first prove the claim in the case $u=0$. We denote by $\n$ the unit exterior normal vector to $\p\Omega$. We extend $\n$ to a smooth vector field in $\R^n$ with compact support in a small neighborhood of $\p\Omega$, and consider $T_\eps\colon \R^n\to \R^n$ given by $T_\eps(x)= x-\sqrt{\eps}\n(x)$.

We consider $\mu$ as a measure in all $\R^n$ with support in $\bar\Omega$ and let $\mu_\eps\coloneqq T_\eps\sharp\mu$ be the push-forward of $\mu$ under $T_\eps$ in the sense that $\mu_\eps(E) = \mu(T^{-1}_\eps(E))$ for any measurable subset $E\subset \bar\Omega$. Then $\mu_\eps$ has support in $\bar \Omega_\eps\coloneqq \{x\in\bar\Omega,\, dist(x,\p\Omega)\ge \sqrt{\eps}\}$.

We cover $\bar\Omega_\eps$ by open disjoint cubes $Q_{i,\eps}=x_{i,\eps}+\eps Q$ centered at $x_{i,\eps}$, where $Q$ is the cube centered at $0$ with sides parallel to the coordinate axes of length $1$. Notice that $Q_{i,\eps}\subset\Omega$ since $\eps<\sqrt{\eps}$. Given some $\phi\in C_c^\infty(B_{1/2})$ we define $u_{i,\eps}\in C_c^\infty(Q_{i,\eps})$ by
$$
u_{i,\eps}(x)=t_{i,\eps}\eps^{-\frac{n-p(x_{i,\eps})}{p(x_{i,\eps})}}\phi\Big(\frac{x-x_{i,\eps}}{\eps}\Big), 
$$
where $t_{i,\eps}\ge 0$ is chosen such that
$$
\int |\nabla u_{i,\eps}|^{p(x)}\,dx = \mu_\eps(Q_{i,\eps}). 
$$
Notice that the scalars $t_{i,\eps}$ are uniformly bounded. We then consider $u_\eps\coloneqq \sum_i u_{i,\eps}\in W^{1,p(x)}_0(\Omega)$. Notice that for any $\eps$, the $u_{i,\eps}$ have disjoint support. As a consequence
$$ 
\int_\Omega |\nabla u_\eps|^{p(x)}\,dx = \sum_i \int_{Q_{i,\eps}} |\nabla u_{i,\eps}|^{p(x)}\,dx = \sum_i \mu_\eps(Q_{i,\eps}) 
= \sum_i \mu_\eps(Q_{i,\eps}\cap \bar \Omega_\eps) = \mu_\eps(\bar \Omega_\eps) = \mu(\bar \Omega)\le 1. 
$$
It follows that $(u_\eps,|\nabla u_\eps|^{p(x)}\,dx)\in\X$ for any $\eps>0$, and that $\{u_\eps\}_{\eps>0}$ is bounded in $W^{1,p(x)}_0(\Omega)$. Morover we have $u_\eps\to 0$ in $L^{p(x)}(\Omega)$. In fact, 
\begin{equation*}
\begin{split}
 \int_\Omega |u_\eps|^{p(x)}\,dx & = \sum_i \int_{Q_{i,\eps}} |u_{i,\eps}|^{p(x)}\,dx
    \le C\sum_i \int_Q \eps^{n(1-\frac{p(x_{i,\eps}+\eps y)}{p(x_{i,\eps})})}\eps^{p(x_{i,\eps}+\eps y)} |\phi|^{p(x_{i,\eps}+\eps y)}\, dy \\
  & \le C\eps^{p^-} \to 0.
\end{split}
\end{equation*}
Hence we can assume that $u_\eps\rightharpoonup 0$ weakly in $L^{q(x)}(\Omega)$. It thus remain to prove that $|\nabla u_\eps|^{p(x)}\,dx\to \mu$ weakly in $M(\bar\Omega)$. Let $\psi\in C(\bar\Omega)$ and $\delta>0$. Then for $\eps>0$ small we have, by the uniform continuity of $\psi$ over $\bar \Omega$, that
\begin{equation*}
\begin{split}
 \int_\Omega \psi |\nabla u_\eps|^{p(x)}\,dx
 & = \sum_i \int_{Q_{i,\eps}} \psi |\nabla u_{i,\eps}|^{p(x)}\,dx
 = \sum_i (\psi(x_{i,\eps})+O(\delta)) \int_{Q_{i,\eps}} |\nabla u_{i,\eps}|^{p(x)}\,dx \\
 & = \sum_i (\psi(x_{i,\eps})+O(\delta)) \mu_\eps(Q_{i,\eps})
  =  \sum_i \psi(x_{i,\eps}) \mu_\eps(Q_{i,\eps}) +  O(\delta) \\
 & = \int_\Omega \psi \,d\mu_\eps +  O(\delta)
  = \int_\Omega \psi\circ T_\eps \,d\mu +  O(\delta)
  = \int_\Omega \psi \,d\mu +  O(\delta).
\end{split}
\end{equation*}

For a general pair $(u,\mu)\in X$, we write $\mu = |\nabla u|^{p(x)} + \tilde \mu + \sum_i \mu_i\delta_{x_i}$, and consider $u_\eps$ such that $(u_\eps,|\nabla u_\eps|^{p(x)})\stackrel{\tau}{\to} (0,\tilde \mu +\sum_i \mu_i \delta_{x_i})$ in $X$ as given by the previous step. Then $\tilde u_\eps = u_\eps + u$ verifies $(\tilde u_\eps,|\nabla \tilde u_\eps|^{p(x)})\stackrel{\tau}{\to} (u,\mu)$ in $X$. Indeed in view of Proposition \ref{desigualdad.nico} and  Remark \ref{remark}, there exists a constant $C>0$ such that 
\begin{equation*}
\begin{split}
  \int_\Omega \Big| |\nabla\tilde u_\eps|^{p(x)} - |\nabla u|^{p(x)} - |\nabla u_\eps|^{p(x)} \Big|\,dx 
 \le C\int_\Omega |\nabla u|^{p(x)-1}|\nabla u_\eps|\,dx 
   + C\int_\Omega |\nabla u||\nabla u_\eps|^{p(x)-1}\,dx.  
\end{split}
\end{equation*}
To prove that the integrals in the rhs goes to $0$, we observe that the sequence $(|\nabla u_\eps|)_\eps$ is bounded in $L^{p(x)}(\Omega)$ and converges to $0$  in $L^1(\Omega)$, so that, up to a subsequence, it also converges to $0$ a.e.. It follows that $|\nabla u_\eps|\to 0$ weakly in $L^{p(x)}(\Omega)$. The convergence to $0$ of the first integral in the rhs follows. The convergence to $0$ of the second integral can be proved in the same way noticing that the sequence  $\{|\nabla u_\eps|^{p(x)-1}\}_{\eps>0}$ converges weakly to 0 in $L^{p(x)'}(\Omega)$ being  bounded in $L^{p(x)'}(\Omega)$ and converging to 0 a.e. As a consequence for any $\psi\in C(\bar\Omega)$, we have
$$  \int_\Omega \psi |\nabla\tilde u_\eps|^{p(x)}\,dx 
 =  \int_\Omega \psi |\nabla u|^{p(x)}\,dx + \int_\Omega \psi |\nabla u_\eps|^{p(x)}\,dx + o(1) 
 = \int_\Omega \psi \,d\mu  + o(1),
$$
i.e. $|\nabla \tilde u_\eps|^{p(x)}\,dx \to \mu $ weakly in $\M(\bar\Omega)$. 
In particular taking $\psi\equiv 1$ gives 
$$  \limsup_{\eps\to 0}\int_\Omega |\nabla\tilde u_\eps|^{p(x)}\,dx \le 1.$$ 
For those $\eps>0$ such that $\int_\Omega |\nabla\tilde u_\eps|^{p(x)}\,dx > 1$, we consider $\delta_\eps>0$, $\delta_\eps\to 0$, such that 
$$
\int_\Omega (1-\delta_\eps)^{p(x)}|\nabla\tilde u_\eps|^{p(x)}\,dx=1,
$$ 
and replace $\tilde u_\eps$ by $(1-\delta_\eps)\tilde u_\eps$. Hence $(\tilde u_\eps,|\nabla\tilde u_\eps|^{p(x)})\in\X$. 

This finishes the proof of the Proposition.
\end{proof}

With the aid of the previous result, we can prove the $\liminf$ property when $\mu$ has no atoms. 

\begin{prop}\label{no atomica}
Let $(u,\mu)\in \X$  such that $\mu$ has no atomic part i.e. $d\mu=|\nabla u|^{p(x)}\, dx + d\widetilde{\mu}$.
Then
$$ \lim_{\varepsilon\to 0}F_\varepsilon(u_\varepsilon,\mu_\varepsilon) = F^*(u,\mu) $$
for every sequence $\{(u_\varepsilon,|\nabla u_\eps|^{p(x)})\}_{\eps>0}\subset \X$ converging to $(u,\mu)$  as $\varepsilon\to 0$.
\end{prop}

\begin{remark} Notice that such a sequence exists in view of the previous  Proposition \ref{Exist}.
\end{remark}

\begin{proof}
Consider a sequence  $\{(u_\varepsilon,\mu_\eps\coloneqq |\nabla u_\eps|^{p(x)})\}_{\eps>0}\subset \X$ converging to $(u,\mu)$. 
According to the CCP, the atoms of the measure $\nu\coloneqq \lim |u_\eps|^{q(x)}\,dx$ (limit in $\M(\bar\Omega)$ - which exists up to a subsequence) are also atoms of $\mu$. 
Since by assumption $\mu$ has no atomic part, we deduce that  $\nu$ also has no atoms, and thus that $u_\varepsilon\to u$ strongly in $L^{q(x)}(\Omega)$ and a.e.. As in the proof of the $\limsup$ property we obtain that 
$$ \limsup_{\varepsilon\to 0}\int_\Omega|u_\varepsilon|^{q(x)-\varepsilon}\,dx \le \int_\Omega|u|^{q(x)}\,dx. $$ 
Moreover using Fatou lemma, 
$$  \liminf_{\varepsilon\to 0}\int_\Omega|u_\varepsilon|^{q(x)-\varepsilon}\,dx \ge \int_\Omega|u|^{q(x)}\,dx. $$ 
Hence
$$ \lim_{\eps\to 0} F_\eps(u_\eps, \mu_\eps) 
 = \lim_{\varepsilon\to 0}\int_\Omega|u_\varepsilon|^{q(x)-\varepsilon}\,dx
 =\int_\Omega|u|^{q(x)}\,dx= F^*(u,\mu),
$$
as we wanted to show.
\end{proof}

We now prove the $\liminf$ property assuming that $\mu$ is purely atomic with a finite number of atoms and total mass strictly less that $1$. 

\begin{prop}\label{atomica}
For every $(u,\mu)\in \X$ of the form $(u,\mu)=(0,\sum_{i=0}^k\mu_i\delta_{x_i})$, with $x_i\in\A$ and $\mu_i>0$ such that
$\mu(\bar\Omega)=\sum_i \mu_i<1$,
there exists $\{(u_\eps,|\nabla u_\eps|^{p(x)})\}_{\eps>0}\subset \X$ converging in $\X$ to $(u,\mu)$ and such that
$$
\lim_{\eps\to 0} F_\eps(u_\eps,|\nabla u_\eps|^{p(x)})=F^*(u,\mu).
$$
\end{prop}

\noindent The proof relies on the following two lemmas. The first one states gives the relation between the two localized Sobolev constant \eqref{defSobLoc} and \eqref{LocSob}. 

\begin{lema}\label{SobLoc2}
For any $x_0\in\A$, 
$$ \widetilde S_{x_0}^{-1}=\bar S_{x_0}^{-q(x_0)}. $$
\end{lema}

\begin{proof}
First, suppose that $\widetilde{S}^{-1}_{x_0}>1$. So there exists $\varepsilon_0>0$ such that
$$
\sup_{u\in \B(B_\varepsilon(x_0))}\int_{B_{\varepsilon}(x_0)}|u|^{q(x)}\,dx>1 \qquad 
 \text{ for all } \eps\leq\eps_0. $$
It follows that
\begin{align*}
\sup_{u\in \B(B_\varepsilon(x_0))} \|u\|_{q(x), B_\varepsilon(x_0)}^{q^-_\eps}
\leq\sup_{u\in \B(B_\varepsilon(x_0))}\int_{B_{\varepsilon}(x_0)}|u|^{q(x)}\,dx 
 \le \sup_{u\in \B(B_\varepsilon(x_0))} \|u\|_{q(x), B_\varepsilon(x_0)}^{q^+_\eps},
\end{align*}
where $q^-_\eps = \inf_{B_\eps(x_0)} q(x)$ and $q^+_\eps = \sup_{B_\eps(x_0)} q(x)$. Notice that
$$ \sup_{u\in \B(B_\varepsilon(x_0))}\|u\|_{q(x), B_\varepsilon(x_0)}
 =\left(\inf_{u\in \tilde \B(B_\eps(x_0))}\|\nabla u\|_{p(x), B_\varepsilon(x_0)}\right)^{-1},
$$
where $\tilde \B(U) = \{u\in W^{1,p(x)}_0(U)\colon \|u\|_{q(x), U}\leq 1\}$. So, recalling that
$$ \lim_{\eps\to 0}\inf_{u\in \tilde \B(B_\eps(x_0))}\|\nabla u\|_{p(x), B_\varepsilon(x_0)} = \bar S_{x_0}, $$
we get
$$ \bar S_{x_0}^{-q(x_0)}=\tilde{S}_{x_0}^{-1}. $$
The case where $\widetilde{S}_{x_0}^{-1}\le 1$ is analogous.
\end{proof}

\begin{lema} \label{lema.construccion}
For any $x_0\in\A$ there exists a sequence $(u_\eps,|\nabla u_\eps|^{p(x)})\in \X$ such that
\begin{equation*}
 (u_\eps,|\nabla u_\eps|^{p(x)})\stackrel{\tau}{\to} (0,\delta_{x_0}),
\end{equation*}
and
\begin{equation*}
 \lim_{\eps\to 0}\int |u_\eps|^{q(x)-\eps}\,dx=\bar S_{x_0}^{-q(x_0)}=\tilde S_{x_0}^{-1}.
\end{equation*}
\end{lema}

\begin{proof}
Let $\delta>0$, there exists $\varepsilon_0$ such that for all $\varepsilon<\varepsilon_0$
$$
|\widetilde{S}_{x_i}^{-1}-\widetilde{S}(p(\cdot),q(\cdot),B_\varepsilon(x_i))^{-1}|<\delta
$$
Moreover, there exists $u_\varepsilon>0$ and $u_\varepsilon\in W_0^{1,p(\cdot)}(B_\varepsilon(x_i))$ such that $\|\nabla u_\varepsilon\|_{p(x)}\leq 1$ and
$$
\widetilde{S}(p(\cdot),q(\cdot),B_\varepsilon(x_i))^{-1}-\delta<\int_{B_\varepsilon(x_i)}|u_\varepsilon|^{q(x)}\,dx=\int_{\Omega}|u_\varepsilon|^{q(x)}\,dx\leq\widetilde{S}(p(\cdot),q(\cdot),B_\varepsilon(x_i))^{-1}
$$
So, we have proved that given $\delta>0$, there exists $u_\delta\in W^{1,p(x)}_0(B_\delta(x_i))$ such that $\|\nabla u_\delta\|_{p(x)} = 1$ and
$$
\left| \widetilde{S}_{x_i}^{-1} - \int_{\Omega}|u_\delta|^{q(x)}\,dx \right| \le \delta
$$
 Observe that the sequence $\{u_\delta\}_{\delta>0}$ verifies
 $$
 u_\delta\to 0\quad \mbox{a.e. in } \Omega,
 $$
and, as
$$
1=\int |\nabla u_\delta|^{p(x)}\,dx;\qquad \text{supp}(|\nabla u_\delta|)\subset B_\delta(x_i),
$$
then $|\nabla u_\delta|^{p(x)}\rightharpoonup\delta_{x_i}$, weakly in the sense of measures, as $\delta\to 0$.

Now, just observe that, by the Lebesgue dominated convergence theorem,
$$
\int_\Omega |u_\delta|^{q(x)-\eps}\, dx \to \int_{\Omega} |u_\delta|^{q(x)}\, dx.
$$

From these facts, the conclusion of the Lemma, follows.
\end{proof}

We can now prove Proposition \ref{atomica}:

\begin{proof}[Proof of Proposition \ref{atomica}]
We prove the claim in the case $k=2$ i.e. for $\mu$ of the form $\mu=\mu_0\delta_{x_0}+\mu_1\delta_{x_1}$ with $x_0,x_1\in\A$ and
$\mu_0,\mu_1>0$, $\mu(\bar\Omega)=\mu_0+\mu_1< 1$.
We denote by $u_{0,\eps}$ and $u_{1,\eps}$ the functions given by the previous proposition corresponding to the points $x_0$ and $x_1$: 
\begin{equation}\label{E}
\begin{split} 
 & (u_{0,\eps},|\nabla u_{0,\eps}|^{p(x)})\stackrel{\tau}{\to} (0,\delta_{x_0}), \qquad 
 \lim_{\eps\to 0}\int |u_{0,\eps}|^{q(x)-\eps}\,dx=\bar S_{x_0}^{-q(x_0)},\\
 & (u_{1,\eps},|\nabla u_{1,\eps}|^{p(x)})\stackrel{\tau}{\to} (0,\delta_{x_1}), \qquad 
 \lim_{\eps\to 0}\int |u_{1,\eps}|^{q(x)-\eps}\,dx=\bar S_{x_1}^{-q(x_1)}.
\end{split} 
\end{equation}

Since $x_0\neq x_1$, the supports of $u_{0,\eps}$ and $u_{1,\eps}$ are disjoint for $\eps$ small.
It follows that the functions 
$$ u_\varepsilon\coloneqq \mu_0^\frac{1}{p(x_0)}u_{\varepsilon,0}+\mu_1^\frac{1}{p(x_1)}u_{\varepsilon,1} $$
satisfy for any given $\psi\in C(\bar\Omega)$ that 
\begin{align*}
\int_\Omega \psi |\nabla u_\varepsilon|^{p(x)}\,dx
 &=\int \mu_0^{\frac{p(x)}{p(x_0)}}\psi |\nabla u_{\varepsilon,0}|^{p(x)}\,dx 
   + \int \mu_1^{\frac{p(x)}{p(x_1)}}\psi |\nabla u_{\varepsilon,1}|^{p(x)}\,dx\\
 & \to \mu_0\psi(x_0)+\mu_1\psi(x_1) = \int\psi\,d\mu. 
\end{align*}
in view of (\ref{E}). In particular $\lim_{\eps\to 0}\int_\Omega  |\nabla u_\varepsilon|^{p(x)}\,dx=\mu_0+\mu_1<1$. 
Hence $(u_\eps,|\nabla u_\eps|^{p(x)})$ belongs to $\X$ for $\eps$ small and converges to $(0,\mu)$ as $\eps\to 0$. 

Moreover 
\begin{align*}
 F_\eps(u_\varepsilon,|\nabla u_\varepsilon|^{p(x)})
&=\int_\Omega|u_\varepsilon|^{q(x)-\varepsilon}\,dx
=\int \mu_0^\frac{q(x)-\varepsilon}{p(x_0)} |u_\varepsilon|^{q(x)-\varepsilon}\,dx
  +\int \mu_1^\frac{q(x)-\varepsilon}{p(x_1)} |u_\varepsilon|^{q(x)-\varepsilon}\,dx \\
&= (1+o(1))\int \mu_0^\frac{q(x)}{p(x_0)} |u_\varepsilon|^{q(x)-\varepsilon}\,dx
  +(1+o(1))\int \mu_1^\frac{q(x)}{p(x_1)} |u_\varepsilon|^{q(x)-\varepsilon}\,dx \\  
& = \mu_0^\frac{q(x_0)}{p(x_0)}\bar S_{x_0}^{-q(x_0)} +\mu_1^\frac{q(x_1)}{p(x_1)}\bar S_{x_1}^{-q(x_1)} + o(1) \\
& =F^*(0,\mu_0\delta_{x_0}+\mu_1\delta_{x_1})+ o(1).
\end{align*}
This finishes the proof of Proposition \ref{atomica} in the case $k=2$. 
The proof when $\mu$ has an arbitrary finite number of atoms is similar.
\end{proof}

The next lemma first proved in \cite{Amar-Garroni} allows to deduce the general case from the two particular cases stated in propositions \ref{no atomica} and \ref{atomica}. Since its proof is identical to that of \cite{Amar-Garroni} and \cite{Palatucci} we omit it. Its statement involves the functional $F^-$ defined in \eqref{defF-}. 

\begin{lema}\label{lema.ppal}
If $F^-(u,\mu)\geq F^*(u,\mu)$ for every $(u,\mu)\in X$ such that
\begin{enumerate}
\item $\mu(\overline{\Omega})<1 $,
\item $\mu=|\nabla u|^{p(x)}+\widetilde{\mu}+\sum_{i=0}^n\mu_i\delta_{x_i}$,
\item $\d(\overline{\text{supp}(|u|+\tilde\mu)}, \bigcup_{i=1}^n \{x_i\}) >0$, 
\end{enumerate}
then $F^-(u,\mu)\geq F^*(u,\mu)$ for every $(u,\mu)\in X$.
\end{lema}

Finally, we can prove  the principal result.

\begin{proof}[Proof of the $\liminf$ inequality]
We only have to check the hypotheses of Lemma \ref{lema.ppal}.
Given some $(u,\mu)\in\X$ as in Lemma \ref{lema.ppal}, we can descompose $\mu$ as $\mu=\mu^1+\mu^2$ with $\mu^1=\sum_{i=0}^n\mu_i\delta_{x_i}$ and $\mu^2=|\nabla u|^{p(x)}+\widetilde{\mu}$. Moreover there exists relatively open subsets $A, B\subset \overline{\Omega}$ such that $\overline{\text{supp}(\mu^1)}\subset\overline{A}$ and $\overline{\text{supp}(|u|+\tilde\mu)}\subset\overline{B}$ and $\overline{A}\cap\overline{B}=\emptyset$.

By propositions \ref{Exist}, \ref{no atomica} and  \ref{atomica}, there exist sequences 
$(u_\varepsilon^1,\mu_\varepsilon^1=|\nabla u_\varepsilon^1|^{p(x)})\in \X$ and 
$(u_\varepsilon^2,\mu_\varepsilon^2=|\nabla u_\varepsilon^2|^{p(x)})\in \X$ with
$u^1_\varepsilon\in W_0^{1,p(x)}(A)$, $u^2_\varepsilon\in W_0^{1,p(x)}(B)$ 
converging in $\X$ to $(0,\mu^1)$ and $(u,\mu^2)$ respectively, and satisfying 
$$  F_\varepsilon(u_\varepsilon^1,\mu_\varepsilon^1)\to F^*(0,\mu^1), \quad \text{ and } 
     F_\varepsilon(u_\varepsilon^2,\mu_\varepsilon^1)\to F^*(0,\mu^2). $$
Consider $u_\varepsilon=u_\varepsilon^1+u_\varepsilon^2$ and $\mu_\varepsilon=\mu_\varepsilon^1+\mu_\varepsilon^2$. 
As $u_\eps^1$ and $u_\eps^2$ have disjoint support, it is easily seen as in the proof of prop.  \ref{atomica}, that $(u_\eps,\mu_\eps)$ belongs to $\X$ and converges to $(u,\mu)$. Moreover 
\begin{align*}
F_\eps(u_\varepsilon,\mu_\varepsilon)&= F_\eps(u_\eps^1, \mu_\eps^1) + F_\eps(u_\eps^2, \mu_\eps^2)\\
&= F^*(0,\mu^1)+F^*(u,\mu^2) + o(1)\\
&=\int_\Omega |u|^{q(x)}\,dx+\sum^n_{i=0}\mu_i^{\frac{p^*(x_i)}{p(x_i)}}\overline{S}_{x_i}^{-p^*(x_i)} + o(1)\\
&=F^*(u,\mu) + o(1).
\end{align*}
This finishes the proof.
\end{proof}

\section{Proof of proposition \ref{SeS} and theorem \ref{u_eps}.}

Before beginning with the proof of proposition \ref{SeS}, we state and prove an easy version of the H\"older inequality, that even it is well known (see e.g. \cite{Diening}) is not the most common version. So we provide here with a proof for the sake of completeness.

\begin{lema}\label{Holder}
Let $f\in L^{p(x)}(\Omega)$ and $g\in L^{p'(x)}(\Omega)$ where $1<p_-\le p(x)\le p_+<\infty$ and $p'(x)=\frac{p(x)}{p(x)-1}$ is conjugate exponent. Then 
\begin{equation}\label{HolderIneq}
 \int_\Omega f(x)g(x)\, dx \le \left(\frac{1}{p_-} + \frac{1}{p'_-}\right) \max\left\{\left(\int_\Omega |f(x)|^{p(x)}\, dx\right)^{1/p_-}; \left(\int_\Omega |f(x)|^{p(x)}\, dx\right)^{1/p_+}\right\} \|g\|_{p'(x)}.
\end{equation} 
\end{lema}

\begin{proof}
Let $\lambda=\|f\|_{p(x)}$ and $\mu=\|g\|_{p'(x)}$. By Young's inequality, we get
\begin{align*}
\int_\Omega \frac{f(x)}{\lambda}\frac{g(x)}{\mu}\, dx &\le \int_\Omega \frac{1}{p(x)}\left(\frac{|f(x)|}{\lambda}\right)^{p(x)}\, dx + \int_\Omega \frac{1}{p'(x)}\left(\frac{|g(x)|}{\mu}\right)^{p'(x)}\, dx\\
&\le \frac{1}{p_-}\int_\Omega \left(\frac{|f(x)|}{\lambda}\right)^{p(x)}\, dx + \frac{1}{p'_-}\int_\Omega \left(\frac{|g(x)|}{\mu}\right)^{p'(x)}\, dx\\
&= \frac{1}{p_-} + \frac{1}{p'_-}
\end{align*}

Now, the result follows just observing that
$$ \lambda = \|f\|_{p(x)} 
  \le  \max\left\{\left(\int_\Omega |f(x)|^{p(x)}\, dx\right)^{1/p_-}; \left(\int_\Omega |f(x)|^{p(x)}\, dx\right)^{1/p_+}\right\}.  $$
\end{proof}

We are now ready to prove Proposition \ref{SeS}.

\begin{proof}[Proof of Proposition \ref{SeS}]
Using H\"older inequality \eqref{HolderIneq}, we have for any $u\in \B(\Omega)$ that 
\begin{align*}
\int_\Omega |u|^{q(x)-\eps}\, dx &\le \left(\frac{1}{\left(\frac{q}{q-\eps}\right)_-} + \frac{1}{\left(\frac{q}{\eps}\right)_-}\right) \left(\int_\Omega |u|^{q(x)}\,dx\right)^{\frac{1}{\left(\frac{q}{q-\eps}\right)_-}} \|1\|_{\left(\frac{q(x)}{\eps}\right)'} \\
&= (1+o(1)) \Big(  \int_\Omega |u|^{q(x)}\,dx \Big)^{1+o(1)}\\
\end{align*}
from which we deduce  that $\limsup_{\eps\to 0} \tilde S^{-1}_\eps\le \tilde S^{-1}$.

For the opposite inequality, we first observe that for any $u\in W^{1,p(x)}_0(\Omega)$, 
$$ \lim_{\eps\to 0} \int_\Omega |u|^{q(x)-\eps}\,dx = \int_\Omega |u|^{q(x)}\,dx. $$ 
In fact, if $u\in W^{1,p(x)}_0(\Omega)$, we can write 
$$ |u|^{q(x)-\eps} = |u|^{q(x)-\eps} 1_{\{|u|\le 1\}} + |u|^{q(x)-\eps} 1_{\{|u|> 1\}}  \le 1 + |u|^{q(x)}, $$
and the result follows by the Dominated Convergence Theorem. 
Now, for a given $\delta>0$, consider $u_\delta\in \B(\Omega)$ such that  
$\int_\Omega |u_\delta|^{q(x)}\ge \tilde S^{-1}-\delta$. Then 
$$ \liminf_{\eps\to 0} \tilde S^{-1}_\eps \ge \liminf_{\eps\to 0} \int_\Omega |u_\delta|^{q(x)-\eps}  
 = \int_\Omega |u_\delta|^{q(x)}\,dx \ge \tilde S^{-1}-\delta. $$
The proof is now complete.
\end{proof}

Before proving Theorem \ref{u_eps} we need the following Sobolev  type inequality deduced from the definition of $\tilde S$: 

\begin{prop}\label{SobIn}
  For any $u\in W^{1,p(x)}_0(\Omega)$, 
 \begin{equation}\label{SobInequ} 
  \int_\Omega |u|^{q(x)}\,dx \le \tilde S^{-1} \max\,\Big\{\|\nabla u\|_p^{q^+},\|\nabla u\|_p^{q^-}\Big\}. 
 \end{equation}
\end{prop} 

\begin{proof} 
Let $u\in W^{1,p(x)}_0(\Omega)$. By definition of the norm $\|\cdot\|_{p(x)}$, 
there holds 
$$   \int_\Omega \Big(\frac{|\nabla u|}{\|\nabla u\|_p}\Big)^{p(x)}\,dx =1.$$
It follows that $v\coloneqq \frac{u}{\|\nabla u\|_p}$ is admissible for $\tilde S^{-1}$ so that 
$$ \int_\Omega \frac{|u|^{q(x)}}{\|\nabla u\|_p^{q(x)}}\,dx \le \tilde S^{-1}. $$
The result follows noticing that $\|\nabla u\|_p^{q(x)}\le \max\,\Big\{\|\nabla u\|_p^{q^+},\|\nabla u\|_p^{q^-}\Big\}$ for a.e. $x\in\Omega$.
\end{proof}

\begin{proof}[Proof of theorem \ref{u_eps}]

Observe that as an immediate consequence of the $\Gamma$-convergence of $F_\eps$ to $F^*$ as stated in theorem \ref{TEO}, we have 
\begin{equation}\label{Estim}
 \lim_{\eps\to 0}\sup_{\X} F_\eps = \sup_{\X} F^*. 
\end{equation} 
Noticing that $\tilde  S_\eps^{-1} = \sup_{\X} F_\eps $, we obtain, in view of the previous proposition, that 
\begin{equation}\label{Estim2}
  \sup_{\X} F^*  = \tilde S^{-1}. 
\end{equation} 

Being subcritical, the embedding $W^{1,p(x)}_0(\Omega)\hookrightarrow L^{q(x)-\eps}(\Omega)$ is compact for any $\eps>0$. 
It follows that there exist an extremal $u_\eps\in\B(\Omega)$ for $\tilde S_\eps^{-1}$ i.e. 
\begin{equation}\label{Extremal}
 \int_\Omega|u_\varepsilon|^{q(x)-\varepsilon}\,dx=\tilde S_\varepsilon^{-1}.  
\end{equation} 
We can assume that $u_\eps\rightharpoonup u$ weakly in $W^{1,p(x)}_0(\Omega)$. 
The $\limsup$ property (\ref{Limsup}) then gives 
$$ 
\limsup_{\varepsilon\to 0}\int_\Omega|u_\varepsilon|^{q(x)-\varepsilon}\,dx  \le F^*(u, \mu)
= \int_\Omega |u|^{q(x)}\,dx+\sum_{i\in I}\mu_i^{\frac{p^*(x_i)}{p(x_i)}}\widetilde{S}_{x_i}^{-1}, 
$$ 
where $\mu$ and the $\mu_i$ are as in the CCP \eqref{CCP}--\eqref{CCP3}. We then obtain in view of \eqref{Estim2}, \eqref{Extremal} and Proposition \ref{SeS} that $(u,\mu)$ is an extremal for $F^*$ i.e.
\begin{equation}\label{E10}
  \int_\Omega |u|^{q(x)}\,dx+\sum_{i\in I}\mu_i^{\frac{p^*(x_i)}{p(x_i)}}\widetilde{S}_{x_i}^{-1}=\tilde S^{-1}.
\end{equation} 

Since $\mathcal{B}(B_\eps(x_0))\subset\mathcal{B}(\Omega)$ for any $x_0\in\bar\Omega$, we see that $\tilde S_{x_0}^{-1}\le \tilde S^{-1}$ for any $x_0\in\A$. Using also the Sobolev inequality (\ref{SobInequ}), we deduce from (\ref{E10}) that 
$$   1  \le   \max\,\Big\{\|\nabla u\|_p^{q^+},\|\nabla u\|_p^{q^-}\Big\}  +\sum_{i\in I}\mu_i^{\frac{p^*(x_i)}{p(x_i)}}.$$  
Moreover since 
\begin{equation}\label{massmu}
  1\ge \mu(\bar\Omega) \ge \int_\Omega |\nabla u|^{p(x)}\,dx + \sum_i\mu_i, 
\end{equation} 
we have 
$$  \max\,\Big\{\|\nabla u_0\|_p^{q^+},\|\nabla u_0\|_p^{q^-}\Big\}  
     = \|\nabla u_0\|_p^{q^-}  
     \le \Big( \int_\Omega |\nabla u_0|^{p(x)}\,dx\Big)^\frac{1}{p^+}.$$ 
It follows that 
$$   1\le   \Big( \int_\Omega |\nabla u_0|^{p(x)}\,dx\Big)^\frac{q^-}{p^+} + \sum_{i\in I} \mu_i^\frac{p^*(x_i)}{p(x_i)}. $$
Since $\frac{q^-}{p^+},\,\frac{p^*(x_i)}{p(x_i)}>1$ for any $i$,  we obtain in view of (\ref{massmu}) that 
$$   1  \le \int_\Omega |\nabla u|^{p(x)}\,dx + \sum_i \mu_i \le 1 $$
where the first  inequality is strict, leading to a contradiction, if one of the terms in the sum belongs to $(0,1)$. 
It follows that 
\begin{enumerate} 
\item[(i)] either $\int_\Omega |\nabla u|^{p(x)}\,dx =0$ and all the $\mu_i$ are $0$ except one $\mu_{i_0}=1$, 
\item[(ii)] or $\mu_i=0$ for any $i\in I$ and $\int_\Omega |\nabla u|^{p(x)}\,dx =1$. 
\end{enumerate} 

In the first case (i), the CCP (\ref{CCP})-(\ref{CCP3}) reduces to 
$$ |u_\eps|^{q(x)}\, dx \stackrel{*}{\rightharpoonup} \nu_{i_0}\,  \delta_{x_{i_0}}, \qquad 
|\nabla u_\eps|^{p(x)}\, dx \stackrel{*}{\rightharpoonup}  \delta_{x_{i_0}}, \qquad 
  \nu_{i_0} \le \tilde S_{x_{i_0}}^{-1}. 
$$
Then using H\"older inequality, 
$$ \tilde S^{-1} = \lim_{\eps\to 0} \int_\Omega |u_\eps|^{q(x)-\eps}\,dx 
                       \le \limsup_{\eps\to 0} \int_\Omega |u_\eps|^{q(x)}\,dx 
                       = \nu_{i_0}\le \tilde S_{x_{i_0}}^{-1}\le\tilde S^{-1}. $$ 
It follows that  $\nu_{i_0} = \tilde S^{-1}$ and we obtain the second alternative in theorem \ref{u_eps}. 

In the second case (ii), it follows from (\ref{E10}) that $u$ is an extremal for $\tilde S^{-1}$. 
Since $u_\eps\to u$ a.e. and $\int_\Omega |u_\eps|^{q(x)}\,dx\to \int_\Omega |u|^{q(x)}\,dx$,  we obtain using the Brezis-Lieb Lemma (see \cite{BL} and also \cite[Lemma 3.4]{FBS}) that 
$$
\int_\Omega |u_\eps-u|^{q(x)}\,dx =\int_\Omega |u_\eps|^{q(x)}\,dx - \int_\Omega |u|^{q(x)}\,dx + o(1)= o(1) 
$$
i.e. $u_\eps\to u$ strongly in $L^{q(x)}(\Omega)$. This ends the proof of Theorem  \ref{u_eps}. 
\end{proof}

%REMARK: para evitar el uso del lema de Brezis-Lieb tambi\'en se puede decir que 
%$$ \int |u_\eps|^{q(x)}\,dx \to \int |u|^{q(x)}\,dx \qquad \text{ y } \qquad |u_\eps|^{q(x)}\,dx\rightharpoonup |u|^{q(x)}\,dx $$ 
%implican que 
%$$ \|u_\eps\|_q\to \|u\|_q $$
%De hecho con $\lambda=\|u\|_q-\delta $ ( con $\delta>0$ dado), tenemos 
%$$ \lim_\eps \int \Big|\frac{u_\eps}{\lambda}\Big|^{q(x)}\,dx = \int \Big|\frac{u}{\lambda}\Big|^{q(x)}\,dx >1 $$ 
%lo que $\|u\|_q-\delta\le \|u_\eps\|_q$. Obtenemos la desigualdad opuesta tomando $\lambda=\|u\|_q+\delta $. 
%
%Luego usamos que $L^{q(x)}(\Omega)$ es uniformemente convexo (ya que $1<q^-<q^+<\infty$) para decir que la convergencia debil + la convergencia de las normas implican la convergencia fuerte (ver Brezis cap 3). Obtenemos entonces que $u_\eps\to u$ fuerte en $L^q$

\section*{Acknowledgements}
This work was partially supported by Universidad de Buenos Aires under grant UBACYT 20020100100400 and by CONICET (Argentina) PIP 5478/1438.

\def\ocirc#1{\ifmmode\setbox0=\hbox{$#1$}\dimen0=\ht0 \advance\dimen0
  by1pt\rlap{\hbox to\wd0{\hss\raise\dimen0
  \hbox{\hskip.2em$\scriptscriptstyle\circ$}\hss}}#1\else {\accent"17 #1}\fi}
  \def\ocirc#1{\ifmmode\setbox0=\hbox{$#1$}\dimen0=\ht0 \advance\dimen0
  by1pt\rlap{\hbox to\wd0{\hss\raise\dimen0
  \hbox{\hskip.2em$\scriptscriptstyle\circ$}\hss}}#1\else {\accent"17 #1}\fi}
\providecommand{\bysame}{\leavevmode\hbox to3em{\hrulefill}\thinspace}
\providecommand{\MR}{\relax\ifhmode\unskip\space\fi MR }
\providecommand{\MRhref}[2]{%
  \href{http://www.ams.org/mathscinet-getitem?mr=#1}{#2}
}
\providecommand{\href}[2]{#2}
\bibliographystyle{plain}
\bibliography{biblio}

\end{document}